\documentclass[12pt]{amsart}

\usepackage{amscd}
\usepackage{verbatim}

\usepackage{amssymb}
\usepackage{pb-diagram,pb-xy}

\usepackage{amssymb}

\usepackage{xy}
\xyoption{all}

\numberwithin{equation}{section}

\DeclareMathAlphabet{\cat}{OT1}{cmss}{m}{sl}

\newtheorem{theorem}[equation]{Theorem}
\newtheorem{proposition}[equation]{Proposition}
\newtheorem{lemma}[equation]{Lemma}
\newtheorem{corollary}[equation]{Corollary}

\newtheorem{ques}[equation]{Question}
\newtheorem*{imp*}{Important}

\theoremstyle{definition}
\newtheorem{remark}[equation]{Remark}
\newtheorem{example}[equation]{Example}
\newtheorem{dfn}[equation]{Definition}
\renewcommand{\(}{\bigl(}
\renewcommand{\)}{\bigr)}

\newcommand{\tens}{\otimes}

\newcommand{\iso}{\stackrel{\sim}{\to}}

\newcommand{\sep}{\mathrm{sep}}
\newcommand{\id}{\mathrm{id}}

\renewcommand{\Im}{\operatorname{Im}}

\newcommand{\Ker}{\operatorname{Ker}}

\newcommand{\ch}{\operatorname{char}}
\newcommand{\Inv}{\operatorname{Inv}}
\newcommand{\Aut}{\operatorname{Aut}}

\newcommand{\res}{\operatorname{res}}

\newcommand{\Br}{\operatorname{Br}}
\newcommand{\Spec}{\operatorname{Spec}}

\newcommand{\SL}{\operatorname{SL}}
\newcommand{\GL}{\operatorname{GL}}
\newcommand{\gSL}{\operatorname{\mathbf{SL}}}

\newcommand{\gSU}{\operatorname{\mathbf{SU}}}
\newcommand{\gSpin}{\operatorname{\mathbf{Spin}}}

\newcommand{\Hom}{\operatorname{Hom}}
\newcommand{\spp}{\operatorname{span}}

\newcommand{\Dyn}{\operatorname{Dyn}}

\newcommand{\Nrd}{\operatorname{Nrd}}

\newcommand{\cone}{\operatorname{cone}}

\newcommand{\fppf}{\operatorname{fppf}}

\def\et{\operatorname{\acute et}}
\newcommand{\xra}{\xrightarrow}

\newcommand{\A}{\mathbb{A}}

\newcommand{\Z}{\mathbb{Z}}

\newcommand{\Q}{\mathbb{Q}}
\newcommand{\R}{\mathbb{R}}

\newcommand{\gm}{\mathbb{G}_m}

\newcommand{\cA}{\mathcal A}

\newcommand{\cG}{\mathcal G}

\DeclareMathOperator{\car}{char}

\newcommand{\injects}{\hookrightarrow}

\newcommand{\qform}[1]{{\left\langle{#1}\right\rangle}}                   % a quadratic form

\newcommand{\F}{\mathbb{F}}

\DeclareMathOperator{\Spin}{Spin}
\DeclareMathOperator{\SO}{SO}
\DeclareMathOperator{\Sp}{Sp}
\DeclareMathOperator{\PGL}{PGL}

\newcommand{\ot}{\otimes}

\newcommand{\nBr}[1]{{{}_{{#1}}\!\Br}}

\newcommand{\tInv}{\widetilde\Inv}

\newcommand{\hInv}{\Inv_h}
\newcommand{\thInv}{\tInv_h}

\newcommand{\tiL}{\stackrel{L}\tens}

\newcommand{\Gm}{\gm}

\usepackage{hyperref}

\title
[Rost invariant on the center, revisited] % colontitle
{Rost invariant on the center, revisited}

\author
{Skip Garibaldi}

\author
{Alexander S. Merkurjev}

\thanks{The work of the second author has been supported by the
NSF grant DMS \#1160206.}

\begin{document}

\begin{abstract}
The Rost invariant of the Galois cohomology of a simple simply connected algebraic group over a field $F$ is defined regardless of the characteristic of $F$, but unfortunately some formulas for it are only known with some hypothesis on the the characteristic.  We improve those formulas by (1) removing the hypothesis on the characteristic and (2) removing an ad hoc pairing that appears in the original formulas.  As a preliminary step of independent interest, we also extend the classification of invariants of quasi-trivial tori to all fields.
\end{abstract}


\maketitle

%%%%%%%%%%%%%%%%%%%%%%%%%%%%%%%%%%%%%%%%%%%%%%%%%%%%%%%%%%%%%%%%
\section{Introduction}

Cohomological invariants provide an important tool to distinguish elements of Galois cohomology groups such as $H^1(F, G)$ where $G$ is a semisimple algebraic group.  In case $G$ is simple and simply connected there are no non-constant invariants with values in $H^d(*, \Q/\Z(d-1))$ for $d < 3$.  For $d = 3$, modulo constants the group of invariants $H^1(*, G) \to H^3(*, \Q/\Z(2))$ is finite cyclic with a canonical generator known as the Rost invariant and denoted by $r_G$; this was shown by Markus Rost in the 1990s and full details can be found in \cite{GMS}.  Rost's theorem raised the questions: How to calculate the Rost invariant of a class in $H^1(F, G)$?  What is a formula for it?

At least for $G$ of inner type $\cat{A}_n$ there is an obvious candidate for $r_G$, which is certainly equal to $mr_G$ for some $m$ relatively prime to $n+1$.  The papers \cite{MPT02} and \cite{GQ07} studied the composition
\begin{equation} \label{rost.ctr}
H^1(F, C) \to H^1(F, G) \xrightarrow{r_G} H^3(F, \Q/\Z(2))
\end{equation}
for $C$ the center of $G$, and under some assumptions on $\car F$, computed the composition in terms of the value of $m$ for type $\cat{A}$.  Eventually the value of $m$ was determined in \cite{GQ11}.  The main result is Theorem \ref{main}, which gives a formula for \eqref{rost.ctr} which does not depend on the type of $G$ nor on $\car F$.  This improves on the results of \cite{MPT02} and \cite{GQ07} by removing the hypothesis on the characteristic and avoiding the ad hoc type-by-type arguments used in those papers.  We do rely on \cite{GQ11} for the computation of $m$ for type $\cat{A}$, but nothing more.

The strategy is to (1) extend the determination of invariants of quasi-trivial tori from \cite{MPT02} to all fields (see Theorem \ref{invariants}), (2) to follow the general outline of \cite{GQ07} to reduce to the case of type $\cat{A}$, and (3) to avoid the ad hoc formulas used in previous work by giving a formula independent of the Killing-Cartan type of $G$. 

Specifically, there is a canonically defined element $t_G^\circ \in H^2(F, C^\circ)$, where $C^\circ$ denotes the dual multiplicative group scheme of $C$ in a sense defined below, and a natural cup product $H^1(F, C) \otimes H^2(F, C^\circ) \to H^3(F, \Q/\Z(2))$.  We prove:
\begin{theorem}\label{main}
Let $G$ be a semisimple and simply connected algebraic group over a field $F$, $C\subset G$ the center of $G$.
Let $t_G^\circ$ be the image of the Tits class $t_G$ under ${\hat\rho}^*: H^2(F,C)\to H^2(F,C^\circ)$.
Then the diagram
\[
\xymatrix{
H^1(F,C) \ar[rd]_{-t_G^\circ\cup}\ar[r]^{i^*} & H^1(F,G) \ar[d]^{r_G} \\
%H^3(F, \Q/\Z(2)) \ar[r]^{\theta^*}
 & H^3(F, \Q/\Z(2))
 }
\]
commutes.
\end{theorem}

This result then gives a general statement for all invariants $H^1(*, G) \to H^3(*, \Q/\Z(2))$, which we state precisely in Theorem \ref{main2} below.

%%%%%%%%%%%%%%%%%%%%%%%%%%%%%%%%%%%%%%%%%%%%%%%%%%%%%%%%%%%%%%%%

\section{Cohomology of groups of multiplicative type} \label{cohomo}

Let $F$ be a field and $M$ a group scheme of multiplicative type over $F$. Then $M$ is uniquely determined by the Galois
module $M^*$ of characters over $F_\sep$. In particular, we have
\[
M(F_\sep)=\Hom(M^*, F_\sep^\times).
\]
If $M$ is a torus $T$, then $T^*$ is a Galois lattice and we set $T_*=\Hom(T^*,\Z)$. We have
\begin{equation}\label{cochar}
T(F_\sep)=T_*\tens F_\sep^\times.
\end{equation}

If $M$ is a finite group scheme $C$ of multiplicative type, we set $C_*:=\Hom(C^*,\Q/\Z)$, so we have a perfect pairing of Galois modules
\begin{equation}\label{pairing}
C_*\tens C^* \to \Q/\Z.
\end{equation}
Write $C^\circ$ for the group of multiplicative type over $F$ with the character module $C_*$. We call $C^\circ$ the group
\emph{dual to $C$}.

\begin{example}
We write $\mu_n$ for the sub-group-scheme of $\Gm$ of $n$-th roots of unity.  The restriction of the natural generator of $\Gm^*$ (the identity $\Gm \to \Gm$) generates $\mu_n^*$ and thereby identifies $\mu_n^*$ with $\Z/n\Z$.  From which $\mu_n^* = \Z/n\Z$ via the pairing \eqref{pairing}, hence $\mu_n^\circ=\mu_n$.
\end{example}

The change-of-sites map $\alpha: \Spec(F)_{\fppf}\to \Spec(F)_{\et}$ yields a functor
\[
\alpha_*:\operatorname{Sh}_{\fppf}(F)\to \operatorname{Sh}_{\et}(F)
\]
between the categories of sheaves over $F$ and
an exact functor
\[
R\alpha_*:\cat{D}^+\operatorname{Sh}_{\fppf}(F)\to \cat{D}^+\operatorname{Sh}_{\et}(F)
\]
between derived categories.

Every group $M$ of multiplicative type can be viewed as a sheaf of abelian groups either in the \'etale or fppf topology.
We have $\alpha_*(M)=M$ for every group $M$ of multiplicative type. If $M$ is smooth, we have
$R^i\alpha_*(M)=0$ for $i>0$ by \cite[Proof of Theorem 3.9]{Milne80}. It follows that $R\alpha_*(M)=M$, hence
\[
H^i_{\et}(F, M)=H^i_{\et}(F, R\alpha_*(M))=H^i_{\fppf}(F, M) \quad \text{for $M$ smooth}.
\]

If
\[
1\to C \to T \to S \to 1
\]
is an exact sequence of algebraic groups with $C$ a finite group of multiplicative type and $T$ and $S$ tori, this sequence
is exact in the fppf-topology but not in the \'etale topology (unless $C$ is smooth). Applying $R\alpha_*$ to the exact triangle
\[
C \to T \to S \to C[1]
\]
in $\cat{D}^+\operatorname{Sh}_{\fppf}(F)$, we get an exact triangle
\[
R\alpha_*(C) \to T(F_\sep) \to S(F_\sep) \to R\alpha_*(C)[1]
\]
in $\cat{D}^+\operatorname{Sh}_{\et}(F)$ since $R\alpha_*(T)=T(F_\sep)$ and the same for $S$. In other words,
\begin{equation}\label{cone}
R\alpha_*(C)=\cone(T(F_\sep)\to S(F_\sep))[-1].
\end{equation}

Recall that $\Z(1)$ is the complex in $\cat{D}^+\operatorname{Sh}_{\et}(F)$ with only one nonzero term $F^\times_\sep$ placed in degree $1$, i.e.,
$\Z(1)=F^\times_\sep[-1]$. Set
\[
C_*(1):=C_*\tiL \Z(1),\quad C^*(1):=C^*\tiL \Z(1),
\]
where the derived tensor product is taken in the derived category $\cat{D}^+\operatorname{Sh}_{\et}(F)$.
If $T$ is an algebraic torus, we write
\[
T_*(1):=T_*\tiL \Z(1)=T_* \tens \Z(1)=T(F_\sep)[-1].
\]

Tensoring the exact sequence
\[
0\to T_* \to S_* \to C_* \to 0
\]
with $\Z(1)$ and using \eqref{cochar}, we get an exact triangle
\[
C_*(1) \to T(F_\sep) \to S(F_\sep) \to C_*(1)[1].
\]
It follows from \eqref{cone} that
\[
C_*(1)=R\alpha_*(C)
\]
and therefore,
\[
H^i_{\fppf}(F,C)=H^i_{\et}(F,R\alpha_*(C))=H^{i}_{\et}(F,C_*(1)).
\]
Recall that we also have
\[
H^i_{\fppf}(F,T)=H^i_{\et}(F,T)=H^{i+1}_{\et}(F,T_*(1)).
\]

\begin{remark}
There is a canonical isomorphism (see \cite[\S 4c]{Merkurjev13})
\[
C_*(1)\simeq C(F_\sep)\oplus (C_*\tens F_\sep^\times)[-1]
\]
The second term in the direct sum vanishes if $\ch(F)$ does not divide the order of $C_*$ or if $F$ is perfect.
\end{remark}

\begin{comment}
\begin{proposition}
There are canonical isomorphisms
\[
H^i(F,C_*(1))\simeq H^i_{\fppf}(F,C) \quad\text{and}\quad H^j(F, C^*(1))\simeq H^j_{\fppf}(F, C^\circ).
\]
\end{proposition}
\end{comment}

\begin{imp*}
To simplify notation we will write $H^i(F,C)$ for $H^i(F,C_*(1))=H^i_{\fppf}(F,C)$ and $H^i(F,C^\circ)$ for $H^i(F,C^*(1))=H^i_{\fppf}(F,C^\circ)$.
\end{imp*}

Every $C$-torsor $E$ over $F$ has a class $c(E)\in H^1(F,C)$.

\begin{example}
%We have $H^2(F,\mu_n)\simeq \nBr{n}(F)$.
Taking colimits of the connecting homomorphism arising from the sequences $1 \to \Gm \to \GL_d \to \PGL_d \to 1$ or $1 \to \mu_d \to \SL_d \to \PGL_d \to 1$ --- which are exact in the fppf topology --- gives isomorphisms $H^2(K, \Gm) \simeq \Br(K)$ and $H^2(K, \mu_n) \simeq \nBr{n}(K)$ as in \cite[4.4.5]{GS06}\footnote{This reference assumes $\car F$ does not divide $n$, because it uses $H^1$ to denote Galois cohomology.  With our notation, their arguments go through with no change.}, which we use.  %These differ by a sign from the isomorphisms given by the cross product construction, see \cite[p.~397]{Book}.
\end{example}

We use the motivic complex $\Z(2)$ of \'etale sheaves over $F$ defined in \cite{Lichtenbaum87} and \cite{Lichtenbaum90}.
Set $\Q/\Z(2):=\Q/\Z\tiL \Z(2)$. The complex $\Q/\Z(2)$ is the direct sum of two complexes. The first
complex is given by the locally constant \'etale sheaf (placed in degree $0$) the
colimit over $n$ prime to $\ch(F)$ of the Galois modules $\mu_{n}^{\tens 2}:=\mu_{n}\tens\mu_n$.
The second complex is nontrivial only in the case
$p=\ch(F)>0$ and it is defined as
\[
\underset{n}{\textrm{colim}}\ W_n\Omega^2_{log}[-2]
\]
with $W_n\Omega^2_{log}$ the sheaf of \emph{logarithmic de Rham-Witt differentials} (see \cite{Kahn96}).

Note that $H^i(F,\Q/\Z(2))\simeq H^{i+1}(F,\Z(2))$ for $i\geq 3$.

Twisting \eqref{pairing} we get the pairings
\begin{gather}
C_*(1)\stackrel{L}\tens C^*(1) \to \Q/\Z(2) \quad \text{and} \notag \\
H^i(F,C)\tens H^j(F, C^\circ)\to H^{i+j}(F,\Q/\Z(2)). \label{fp1}
\end{gather}

If $S$ is a torus over $F$, we have $S_*(1)=S_*\tens\gm[-1]=S[-1]$ and the pairings
\begin{gather}
S_*\tens S^* \to \Z,\quad S_*(1)\tiL S^*(1) \to \Z(2) \quad \text{and} \notag \\
H^i(F,S)\tens H^j(F, S^\circ)\to H^{i+j+2}(F,\Z(2))=H^{i+j+1}(F,\Q/\Z(2)) \label{fp2}
\end{gather}
if $i+j\geq 2$.

Let
\[
1\to C \to T \to S \to 1
\]
be an exact sequence with $T$ and $S$ tori and $C$ finite. Dualizing we get an exact sequence of dual groups
\begin{equation} \label{dual.seq}
1\to C^\circ \to S^\circ \to T^\circ \to 1.
\end{equation}
We have the homomorphisms
\[
\varphi:S(F)\to H^1(F,C),\quad\quad\psi: H^2(F,C^\circ)\to H^2(F,S^\circ).
\]

\begin{proposition}\label{commute}
For every $a\in S(F)$ and $b\in H^2(F,C^\circ)$, we have $\varphi(a)\cup b = a \cup\psi(b)$ in $H^3(F,\Q/\Z(2))$.
Here the cup products are taken with respect to the pairings \eqref{fp1} and \eqref{fp2} respectively.
\end{proposition}
\begin{proof}
The pairing $S_*\tens S^* \to \Z$ extends uniquely to a pairing $S_*\tens T^* \to \Q$. We have then a morphism of exact triangles
\[
\xymatrix{
S_*(1)\tiL S^*(1)\ar[d]\ar[r] & S_*(1)\tiL T^*(1)\ar[d]\ar[r] & S_*(1)\tiL C^*(1)\ar[d]^s\ar[r] & S_*(1)\tiL S^*(1)[1]\ar[d] & \\
\Z(2) \ar[r] & \Q(2) \ar[r] & \Q/\Z(2) \ar[r] & \Z(2)[1]
 }
\]
and a commutative diagram
\[
\xymatrix{
H^1(F,S_*(1))\tens H^2(F,C^*(1))\ar[d]\ar[r] & H^1(F,S_*(1))\tens H^2(F,S^*(1)[1])\ar[d]  \\
H^3(F,\Q/\Z(2)) \ar[r] & H^3(F,\Z(2)[1])
 }
\]
and therefore, a commutative diagram
\[
\xymatrix{
H^0(F,S)\tens H^2(F,C^\circ)\ar[d]\ar[r] & H^0(F,S)\tens H^2(F,S^\circ)\ar[d]  \\
H^3(F,\Q/\Z(2)) \ar@{=}[r] & H^3(F,\Q/\Z(2)).
 }
\]
On the other hand, the composition $S_*(1)\tiL C^*(1)\to C_*(1)\tiL C^*(1)\to \Q/\Z(2)$ coincides with $s$. Therefore, we have a commutative diagram
\[
\xymatrix{
H^1(F,S_*(1))\tens H^2(F,C^*(1))\ar[d]\ar[r] & H^1(F,C_*(1))\tens H^2(F,C^*(1))\ar[d]  \\
H^3(F,\Q/\Z(2)) \ar@{=}[r] & H^3(F,\Q/\Z(2))
 }
\]
and therefore, a diagram
\[
\xymatrix{
H^0(F,S)\tens H^2(F,C^\circ)\ar[d]\ar[r] & H^1(F,C)\tens H^2(F,C^\circ)\ar[d]  \\
H^3(F,\Q/\Z(2)) \ar@{=}[r] & H^3(F,\Q/\Z(2)).
 }
\]
The result follows.
\end{proof}

\begin{remark} \label{signs}
We have used that the diagram
\[
\xymatrix{
H^i(A[a])\tens H^j(B[b])\ar@{=}[d]\ar[r] & H^{i+j}(A[a]\tens B[b])\ar@{=}[d]  \\
H^{i+a}(A)\tens H^{j+b}(B) \ar[r] & H^{i+j+a+b}(A\tens B)
 }
\]
is $(-1)^{ib}$-commutative for all complexes $A$ and $B$.
\end{remark}

Let $A$ be an \'etale algebra over $F$ and $C$ a finite group scheme of multiplicative type over $A$. Then $C':=R_{A/F}(C)$
is a finite group of multiplicative type over $F$. Moreover, ${C'}^\circ\simeq R_{A/F}(C^\circ)$ and there are canonical isomorphisms
\[
\iota: H^i(A,C) \iso H^i(F,C') \quad\text{and}\quad \iota^\circ: H^i(A,C^\circ) \iso  H^i(F,{C'}^\circ).
\]

\begin{lemma}\label{compac}
We have $\iota(x)\cup\iota^\circ(y)=N_{A/F}(x\cup y)$ in $H^{i+j}(F,\Q/\Z(2))$ for every $x\in H^i(A,C)$ and $y\in H^j(A,C^\circ)$.
\end{lemma}

\begin{proof}
The group scheme $C'_A$ is naturally isomorphic to the product $C_1\times C_2\times\cdots\times C_s$ of group schemes over $A$
with $C_1=C$. Let $\pi: C'_A \to C$ be the natural projection. Similarly,
${C'}^\circ\simeq C_1^\circ\times C_2^\circ\times\cdots\times C_s^\circ$. Write $\varepsilon: C^\circ\to {C'_A}^\circ$
for the natural embedding. Then the inverse of $\iota$ coincides with the composition
\[
H^i(F,C')\xra{\res} H^i(A,C'_A)\xra{\pi^*} H^i(A,C)
\]
and $\iota^\circ$ coincides with the composition
\[
H^i(A,C^\circ) \xra{\varepsilon^*} H^i(A,{C'}^\circ_A) \xra{N_{A/F}} H^i(F,{C'}^\circ).
\]
Since $\pi^*(\iota(x))=x$, we have $\res(\iota(x))=(x, x_2,\dots x_s)$ for some $x_i$. On the other hand, $\varepsilon^*(y)=(y,0,\dots,0)$, hence
\begin{equation}\label{bbb}
\res(\iota(x))\cup\varepsilon^*(y)=x\cup y.
\end{equation}
Finally,
\begin{align*}
\iota(x)\cup\iota^\circ(y)  & = \iota(x)\cup N_{A/F}(\varepsilon^*(y))  \\
& = N_{A/F}\(\res(\iota(x))\cup \varepsilon^*(y)\)  \quad\text{by the projection formula} \\
& = N_{A/F}(x\cup y)  \quad\text{by \eqref{bbb}.}  \qedhere\\
\end{align*}
\end{proof}

\begin{lemma}[Projection formula]\label{1}
Let $f:C\to C'$ be a homomorphism of finite group schemes of multiplicative type.  For $a\in H^m(F, C)$,
the diagram
\[
\xymatrix{
H^k(F, {C'}^\circ) \ar[d]_{f^*}\ar[r]^-{\cup f_*(a)} & H^{k+m}(F,\Q/\Z(2)) \ar@{=}[d] \\
H^k(F, C^\circ) \ar[r]^-{\cup a} & H^{k+m}(F,\Q/\Z(2))
 }
\]
commutes.
\end{lemma}

\begin{proof}
The pairing used in the diagram are induced by the pairings $C^*\tens C_*\to\Q/\Z$
and ${C'}^*\tens C'_*\to\Q/\Z$. The (obvious) projection formula for these pairing reads $\langle f^*(x),y\rangle=\langle x,f_*(y)\rangle$
for $x\in {C'}^*$ and $y\in C_*$.
\end{proof}

%%%%%%%%%%%%%%%%%%%%%%%%%%%%%%%%%%%%%%%%%%%%%%%%%%%%%%%%%%%%%%%%

\section{Invariants of quasi-trivial tori}

\subsection{Cohomological invariants}
For a field $F$ write $H^j(F)$ for the cohomology group $H^j(F,\Q/\Z(j-1))$, where $j\geq 1$  (see \cite{GMS}).
In particular, $H^1(F)$ is the character group of continuous homomorphisms $\Gamma_F \to \Q/\Z$ and $H^2(F)$ is the Brauer group $\Br(F)$.

The assignment $K\mapsto H^j(K)$ is functorial with respect to arbitrary field extensions. If $K'/K$ is a finite field
extension, we have a well-defined \emph{norm map} $N_{K'/K}: H^j(K')\to H^j(K)$.

The graded group $H^*(F)$ is a (left) module over the Milnor ring $K_*(F)$.

\begin{dfn} \label{inv.def}
Let $\cA$ be a functor from the category of field extensions of $F$ to pointed sets.
A degree $d$ \emph{cohomological invariant of $\cA$} is a collection of maps of pointed sets
\[
\iota_K \!: \cA(K)\to H^d(K)
\]
for all field extensions $K/F$, functorial in $K$.
The degree $d$ cohomological invariants of $\cG$ form an abelian group
denoted by $\Inv^d(\cA)$. If $L/F$ is a field extension, we have a \emph{restriction homomorphism}
\[
\Inv^d(\cA)\to \Inv^d(\cA_L),
\]
where $\cG_L$ is the restriction of $\cG$ to the category of field extensions of $L$.

If the functor $\cA$ factors through the category of groups, we further consider the subgroup $\hInv^d(\cA)$ of $\Inv^d(\cA)$ consisting of those invariants $\iota$ such that $\iota_K$ is a group homomorphism for every $K$.
\end{dfn}

\begin{example}
If $G$ is an algebraic group over $F$, we can view $G$ as a functor taking a field extension $K$ to the group $G(K)$
of $K$-points of $G$; in this case we consider $\hInv^d(G)$. We have also another functor $H^1(G) \!: K \to H^1(K, G)$ and we consider $\Inv^d(G)$.  If $G$ is commutative, then $H^1(K, G)$ is a group for every $K$, and we also consider $\hInv^d(H^1(G))$.
\end{example}

\subsection{Residues}
Our goal is to prove Theorem \ref{invariants} concerning the group $\hInv^d(T)$ for $T$ a quasi-split torus. Such invariants
of order not divisible by $\ch(F)$ we determined in \cite{MPT02}.  We modify the method from \cite{MPT02} so that it works in general.  The difficulty is that the groups $H^j(K)$
do not form a cycle module, because the residue homomorphisms need not exist. 

If $K$ is a field with discrete valuation $v$ and residue field $\kappa(v)$, write $H^j(F)_{nr,v}$ for the subgroup of all elements
of $H^j(F)$ that are split by finite separable extensions $K/F$ such that $v$ admits an unramified extension to $K$. Note that every
element in $H^j(F)_{nr,v}$ of order not divisible by $\ch(F)$ belongs to $H^j(F)_{nr,v}$.

There are
\emph{residue homomorphisms} (see \cite{GMS} or \cite{Kato82})
\[
\partial_v: H^j(K)_{nr,v}\to H^{j-1}(\kappa(v)).
\]
\begin{example}\label{simplest}
Let $K=F(t)$ and let $v$ be the discrete valuation associated with $t$. Then $\kappa(v)=F$ and $\partial_v(t\cdot h_K)=h$
for all $h\in H^{j-1}(F)$.
\end{example}

\begin{lemma}\label{unram}
Let $K'/K$ be a field extension and let $v'$ be a discrete valuation on $K'$ unramified over its restriction $v$ on $K$. Then the diagram
\[
\xymatrix{
 H^j(K)_{nr,v} \ar[d]\ar[r]^-{\partial_v} & H^{j-1}(\kappa(v)) \ar[d]\\
 H^j(K')_{nr,v'} \ar[r]^-{\partial_{v'}} & H^{j-1}(\kappa(v'))
 }
\]
commutes.
\end{lemma}

%We modify the method \cite{MPT02} so that it works in general.

\subsection{Invariants of tori}
Let $A$ be an \'etale $F$-algebra and $T^A$ the corresponding quasi-split torus, i.e.,
\[
T^A(K)=(A\tens_F K)^\times
\]
for every field extension $K/F$. If $B$ is another \'etale $F$-algebra, then
\[
T^{A\times B}=T^A\times T^B
\]
and
\[
\hInv^d(T^{A\times B})\simeq \hInv^d(T^{A})\oplus \hInv^d(T^{B}).
\]

Write $A$ as a product of fields: $A=L_1\times L_2\times\cdots\times L_s$.
Set 
\[
H^i(A):=H^i(L_1)\oplus H^i(L_2)\oplus\cdots\oplus H^i(L_s).
\]

For $d\geq 2$ define a homomorphism
\[
\alpha^A:H^{d-1}(A)\to \hInv^d(T^A)
\]
as follows. If $h\in H^{d-1}(A)$, then the invariant $\alpha^A(h)$ is defined by
\[
\alpha^A(h)(t)=N_{A\tens K/K}(t\cdot h_{A\tens K})\in H^d(K)
\]
for a field extension $K/F$ and $t\in T^A(K)=(A\tens_F K)^\times$.

\begin{remark} \label{restate}
In the notation of the previous section, $(T^A)^\circ \simeq T^A$, and we have
\[
H^{d-1}(F,(T^A)^\circ)=H^{d-1}(F,T^A)=H^{d-1}(A,\gm)=H^{d-1}(A).
\]
The pairing \eqref{fp2} for the torus $T^A$, $i = 0$, and $j = 2$,
\[
A^\times\tens H^2(A)=T^A(F)\tens H^2(F,(T^A)^\circ)\to H^3(F),
\]
takes $t\tens h$ to $N_{A/F}(t\cup h_A)=\alpha^A(h)(t)$. In other words, the map $\alpha^A$ coincides with the map
\[
H^2(F,(T^A)^\circ)\to \hInv^3(T^A)
\]
given by the cup product.
\end{remark}

Note that every element $h\in H^{d-1}(A)$ is split by an \'etale extension of $A$, hence the invariant $\alpha^A(h)$ vanishes
when restricted to $F_\sep$.

\begin{ques}
Do all invariants in $\hInv^d(T^A)$ vanish when restricted to $F_\sep$?  
\end{ques}

The answer is ``yes'' when $\car F = 0$.  Indeed, for any prime $p \ne \car F$ and for $F$ separably closed, the zero map is the only invariant $T^A(*) \to H^d(*, \Q_p/\Z_p(d - 1))$ that is a homomorphism of groups \cite[Prop.~2.5]{Merkurjev99}. 

The main result of this section is:

\begin{theorem}\label{invariants}
The sequence
\[
0\to H^{d-1}(A)\xra{\alpha^A} \hInv^d(T^A)\xra{\res} \hInv^d(T^A_\sep)
\]
is exact.
\end{theorem}

That is, defining $\thInv^d(T^A) := \ker \res$, we claim that $\alpha^A \!: H^{d-1}(A) \iso \thInv^d(T^A)$.

The torus $T^A$ is embedded into the affine space $\A(A)$ as an open set. Let $Z^A$ be the closed
complement $\A(A)\setminus T^A$ and let $S^A$ be the smooth locus of $Z^A$ (see \cite{MPT02}). Then $S^A$ is a smooth scheme over $A$. In fact,
$S^A$ is a quasi-split torus over $A$ of the $A$-algebra $A'$ such that $A\times A'\simeq A\tens_F A$.
We have $A=L_1\times L_2\times\cdots\times L_s$,
where the $L_i$'s are finite separable field extensions of $F$, and the connected components of $S^A$ (as well as the irreducible components of $Z^A$)
are in $1$-$1$ correspondence with the factors $L_i$. Let $v_i$ for $i=1,2,\dots, s$ be the discrete valuation of the function field $F(T^A)$ corresponding to the $i^{th}$ connected
component $S_i$ of $S^A$, or equivalently, to the $i^{th}$ irreducible component $Z_i$ of $Z^A$. The residue field of $v_i$ is equal to the
function field $F(Z_i)=F(S_i)$.
We then have the residue homomorphisms
\[
\partial_i: H^d(F(T^A))_{nr,v_i}\to H^{d-1}(F(Z_i))=H^{d-1}(F(S_i)).
\]
Write $\widetilde H^d(F(T^A))$ for the kernel of the natural homomorphism $H^d(F(T^A))\to H^d(F_\sep(T^A))$. Since every extension of the valuation $v_i$ to
$F_\sep(T^A)$ is unramified, we have $\widetilde H^d(F(T^A))\subset H^d(F(T^A))_{nr,v_i}$ for all $i$. Write $F(S^A)$ for the product of $F(S_i)$ over all $i$.
The sum of the restrictions of the maps $\partial_i$ on $\widetilde H^d(F(T^A))$ yields a homomorphism
\[
\partial^A: \widetilde H^d(F(T^A))\to H^{d-1}(F(S^A)).
\]

Applying $u \in \thInv^d(T^A)$ to the generic element $g_{\operatorname{gen}}$ of $T^A$ over
the function field $F(T^A)$, we get a cohomology class $u(g_{\operatorname{gen}})\in H^d(F(T^A))$. By assumption on $u$, we have
$u(g_{\operatorname{gen}})\in \widetilde H^d(F(T^A))$. Applying $\partial^A$, we get a homomorphism
\[
\beta^A: \thInv^d(T^A)\to H^{d-1}(F(S^A)),\quad u\mapsto \partial^A(u(g_{\operatorname{gen}})).
\]

If $B$ is another \'etale $F$-algebra. We have (see \cite{MPT02})
\[
S^{A\times B}=S^A\times T^B + T^A\times S^B.
\]
In particular, $F(S^A)\subset F(S^{A\times B})\supset  F(S^B)$.
Lemma \ref{unram} then gives:

\begin{lemma}\label{com1}
The diagram
\[
\xymatrix{
 \thInv^d(T^A)\oplus \thInv^d(T^B) \ar@{=}[d]\ar[r]^-{\beta^A\oplus\beta^B} & H^{d-1}(F(S^A))\oplus H^{d-1}(F(S^B)) \ar[d]\\
 \thInv^d(T^{A\times B}) \ar[r]^-{\beta^{A\times B}} & H^{d-1}(F(S^{A\times B}))
 }
\]
commutes.
\end{lemma}

Recall that $S^A$ is a smooth scheme over $A$ with an $A$-point. It follows that $A\subset F(S^A)$ and the natural homomorphism
\[
H^j(A)\to H^j(F(S^A))
\]
is injective by \cite[Proposition A.10]{GMS}. We shall view $H^j(A)$ as a subgroup of $H^j(F(S^A))$.

Let $A=L_1\times L_2\times\cdots\times L_s$ be the decomposition of an \'etale $F$-algebra $A$ into a product of fields. The
\emph{height} of $A$ is the maximum of the degrees $[L_i:F]$. The height of $A$ is $1$ if and only if $A$ is split. The
following proposition will be proved by induction on the height of $A$.

\begin{proposition}\label{smaller}
The image of the homomorphism $\beta^A$ is contained in $H^{d-1}(A)$.
\end{proposition}
\begin{proof}
By Lemma \ref{com1} we may assume that $A=L$ is a field. If $L=F$, we have $S^A=\Spec F$, so $A=F(S^A)$ and the statement is clear.

Suppose $L\neq F$.
The algebra $L$ is a canonical direct factor of $L\tens_F L$. It follows that
the homomorphism $\beta^L$ is a direct summand of $\beta^{L\tens L}$. Since the height of the $L$-algebra $L\tens_F L$ is less than the
height of $A$, by the induction hypothesis, $\Im(\beta^{L\tens L})\subset H^{d-1}(L\tens L)$. It follows that
$\Im(\beta^{L})\subset H^{d-1}(L)$.
\end{proof}

It follows from Proposition \ref{smaller} that we can view $\beta^A$ as a homomorphism
\[
\beta^A: \thInv^d(T^A)\to H^{d-1}(A).
\]

We will show that $\alpha^A$ and $\beta^A$ are isomorphisms inverse to each other. First consider the simplest case.

\begin{lemma}\label{simmm}
The maps $\alpha^A$ and $\beta^A$ are isomorphisms inverse to each other in the case $A=F$.
\end{lemma}

\begin{proof}
If $A=F$, we have $T^A=\gm$. The generic element $g_{\operatorname{gen}}$ is equal to $t\in F(t)^\times=F(\gm)$.
Let $h\in H^{d-1}(A)=H^{d-1}(F)$. Then the invariant $\alpha^F(h)$ takes $t$ to $t\cdot h\in\widetilde H^d(F(t))$.
By example \ref{simplest}, $\beta^F(\alpha^F(h))=\partial_v(t\cdot h)=h$, i.e., the composition $\beta^F\circ\alpha^F$
is the identity. It suffices to show that $\alpha^F$ is surjective.

Take $u\in\thInv^d(\gm)$. We consider $t$ as an
element of the complete field $L:=F((t))$ and let $x=u_L(t)\in H^d(L)$. By assumption, $x$ is split by the maximal unramified extension
$L':=F_\sep((t))$ of $L$. By a theorem of Kato \cite{Kato82},
\[
x\in\Ker\big(H^d(L)\to H^d(L')\big)=H^d(F)\oplus t\cdot H^{d-1}(F),
\]
i.e., $x=h'_L+ t\cdot h_L$ for some $h'\in H^d(F)$ and $h\in H^{d-1}(F)$.

Let $K/F$ be a field extension. We want to compute $u_K(a)\in H^d(K)$ for an element $a\in K^\times$.
Consider the field homomorphism $\varphi:L\to M:=K((t))$ taking a power series $f(t)$ to $f(at)$. By functoriality,
\[
u_M(at)=u_M(\varphi(t))=\varphi_*(u_L(t))=\varphi_*(x)=\varphi_*(h'_L+ t\cdot h_L)=h'_M+(at)\cdot h_M,
\]
therefore,
\[
u_M(a)=u_M(at)-u_M(t)=(h'_M+(at)\cdot h_M)-(h'_M+t\cdot h_M)=a\cdot h_M.
\]
It follows that $u(a)=a\cdot h_K$ since the homomorphism $H^d(K)\to H^d(M)$ is injective by \cite[Proposition A.9]{GMS}.
We have proved that $u=\alpha^A(h)$, i.e., $\alpha^A$ is surjective.
\end{proof}

\begin{lemma}\label{com2}
The homomorphism $\beta^A$ is injective.
\end{lemma}
\begin{proof}
The proof is similar to the proof of Proposition \ref{smaller}. We induct on the height of $A$. The right vertical homomorphism in
Lemma \ref{com1} is isomorphic to the direct sum of the two homomorphisms $H^{d-1}(F(S^A))\to H^{d-1}(F(S^A\times T^B))$ and
$H^{d-1}(F(S^B))\to H^{d-1}(F(T^A\times S^B))$. Both homomorphisms are injective by \cite[Proposition A.10]{GMS}. It follows from
Lemma \ref{com1} that we may assume that $A=L$ is a field.

The case $L=F$ follows from Lemma \ref{simmm}, so we may assume that $L\neq F$. The homomorphism $\beta^L$ is a direct summand of
$\beta^{L\tens L}$. The latter is injective by the induction hypothesis, hence so is $\beta^L$.
\end{proof}

\begin{lemma}\label{com3}
The composition $\beta^A\circ\alpha^A$ is the identity.
\end{lemma}
\begin{proof}
We again induct by the height of $A$. By Lemma \ref{com1} that we may assume that $A=L$ is a field.

The case $L=F$ follows from Lemma \ref{simmm}, so we may assume that $L\neq F$. The homomorphisms $\alpha^L$ and $\beta^L$ are direct summands of
$\alpha^{L\tens L}$ and $\beta^{L\tens L}$ respectively. The composition $\beta^{L\tens L}\circ\alpha^{L\tens L}$ is the identity by the induction hypothesis,
hence $\beta^A\circ\alpha^A$ is also the identity.
\end{proof}

It follows from Lemma \ref{com2} and Lemma \ref{com3} that $\alpha^A$ and $\beta^A$ are isomorphisms inverse to each other.
This completes the proof of Theorem \ref{invariants}.

%%%%%%%%%%%%%%%%%%%%%%%%%%%%%%%%%%%%%%%%%%%%%%%%%%%%%%%%%%%%%%%%%%%%%
\section{Invariants of groups of multiplicative type}

In this section, $C$ denotes a group of multiplicative type over $F$ such that there exists an exact sequence
\[
1 \to C \to T \to S \to 1
\]
such that $S$ and $T$ are quasi-trivial tori.  For example, this holds if $C$ is the center of a simply connected semisimple group $G$ over $F$, such as $\mu_n$. In that case, $C$ is isomorphic to the center of the quasi-split inner form $G^q$ of $G$, and we take $T$ to be any quasi-trivial maximal torus in $G^q$.  Then $T^*$ is the weight lattice $\Lambda_w$ and $S^* \simeq \Lambda_r$, where the Galois action permutes the fundamental weights and simple roots respectively.

\begin{proposition}\label{invh1}
Every invariant in $\thInv^3(H^1(C))$ is given by the cup product via the pairing \eqref{fp1} with a
unique element in $H^2(F,C^\circ)$.
\end{proposition}

\begin{proof}
Since $H^1(K, T) = 1$ for every $K$, the connecting homomorphism $S(K) \to H^1(K, C)$ is surjective for every $K$ and therefore the natural homomorphism
\[
\hInv^3(H^1(C)) \to \hInv^3(S)
\]
is injective.

Consider the diagram
\[
\xymatrix{
H^2(F,C^\circ)\ar[d]\ar@{^{(}->}[r] & H^2(F,S^\circ)\ar[d]^{\wr}\ar[r] & H^2(F,T^\circ) \ar[d]^{\wr}\\
\thInv^3(H^1(C)) \ar@{^{(}->}[r] & \thInv^3(S) \ar[r] & \thInv^3(T),
 }
\]
where the vertical homomorphisms are given by cup products and the top row comes from the exact sequence \eqref{dual.seq}; it is exact since $H^1(K, T^\circ) = 1$ for every field extension $K/F$.  The bottom row comes from applying $\thInv^3$ to the sequence $T(K) \to S(K) \to H^1(K, C)$; it is a complex.   The vertical arrows are cup products, and the middle and right ones are isomorphisms by Theorem \ref{invariants} and Remark \ref{restate}.
The diagram commutes by Proposition \ref{commute}.  By diagram chase, the left vertical map is an isomorphism.
\end{proof}

Note that the group $H^2(F,T)$ is a direct sum of the Brauer groups of finite extensions of $F$. Therefore, we have the following, a coarser version of   \cite[Prop.~7]{G:outer}:

\begin{lemma}\label{inj}
The homomorphism $H^2(F,C)\to\coprod \Br(K)$, where the direct sum is taken over all field extensions $K/F$ and all characters
of $C$ over $K$, is injective.
\end{lemma}

\begin{remark}
The group $G$ becomes quasi-split over the function field $F(X)$ of the variety $X$ of Borel subgroups of $G$, so $F(X)$ kills $t_G$.  But the kernel of $H^2(F, C) \to H^2(F(X), C)$ need not be generated by $t_G$, as can be seen by taking $G$ of inner type $D_n$ for $n$ divisible by 4.
\end{remark}

%%%%%%%%%%%%%%%%%%%%%%%%%%%%%%%%%%%%%%%%%%%%%%%%%%%%%%%%%%%%%%%%
\section{Root system preliminaries}\label{root}

\subsection{Notation}
Let $V$ be a real vector space and $R\subset V$ a root system (which we assume is reduced). Write $\Lambda_r\subset \Lambda_w$ for the \emph{root} and \emph{weight lattices} respectively.
For every root $\alpha\in R$, the reflection $s_\alpha$ with respect to $\alpha$ is given by the formula
\begin{equation}\label{ref}
s_\alpha(x)=x-\alpha^\vee(x) \cdot \alpha,
\end{equation}
for every $x\in V$, where $\alpha^\vee\in V^*:=\Hom_{\R}(V,\R)$ is the \emph{coroot} dual to $\alpha$. Write $R^\vee\subset V^*$ for the dual root system and $\Lambda_r^\vee\subset \Lambda_w^\vee$
for the corresponding lattices. We have
\[
\Lambda_r^\vee=(\Lambda_w)^*:=\Hom(\Lambda_w,\Z) \quad\text{and}  \quad\Lambda_w^\vee=(\Lambda_r)^*.
\]

The \emph{Weyl group} $W$ of $R$ is a normal subgroup of the automorphism group $\Aut(R)$ of the root system $R$.
The factor group $\Aut(R)/W$ is isomorphic to the automorphism group $\Aut(\Dyn(R))$ of the Dynkin diagram of $R$.
There is a unique $\Aut(R)$-invariant scalar product $(\ ,\ )$ on $V$ normalized so that square-length
$d_{\alpha^\vee}:=(\alpha,\alpha)$ of short roots in every irreducible component of $R$ is equal to $1$.
The formula \eqref{ref} yields an equality
\[
\alpha^\vee(x)=\frac{2(\alpha, x)}{(\alpha,\alpha)}
\]
for all $x\in V$ and $\alpha\in R$.

We may repeat this construction with the dual root system $R^\vee$, defining $(,)^\vee$ on $V^*$ so that the square-length $d_\alpha := (\alpha^\vee, \alpha^\vee)^\vee$ is 1 for short coroots $\alpha^\vee$ (equivalently, long roots $\alpha$).

\subsection{The map $\varphi$}

\begin{proposition}\label{phi}  
There is a unique $\R$-linear map $\varphi \!: V^* \to V$ such that $\varphi(\alpha^\vee) = \alpha$ for all short $\alpha^\vee$.  Furthermore, $\varphi$ is $\Aut(R)$-invariant, $\varphi(\alpha^\vee) = d_\alpha \alpha$ for all $\alpha^\vee \in R^\vee$, $\varphi(\Lambda_w^\vee) \subseteq \Lambda_w$, and $\varphi(\Lambda_r^\vee) \subseteq \Lambda_r$.  Analogous statements hold for $\varphi^\vee \!: V \to V^*$.  If $R$ is irreducible, then $\varphi \varphi^\vee \!: V^* \to V^*$ and $\varphi^\vee \varphi \!: V \to V$ are multiplication by $d_\alpha$ for $\alpha$ a short root.
%There is a unique $\Aut(R)$-equivariant linear isomorphism $\varphi: V\to V^\vee$ such that
%$\varphi(\alpha)=d_\alpha\cdot\alpha^\vee$ for every $\alpha\in R$.
%Moreover, $\varphi(\Lambda_r)\subset \Lambda_r^\vee$ and $\varphi(\Lambda_w)\subset \Lambda_w^\vee$.
\end{proposition}

\begin{proof}
Define $\varphi^\vee$ by $\qform{\varphi^\vee(x), y}=2(x,y)$ for $x, y \in V$ and $\varphi$ by $\qform{x', \varphi(y')} = 2(x', y')^\vee$ for $x', y' \in V^*$.
%In other words, $\varphi=2\psi$,
%where the isomorphism $\psi:V\iso V^*$ is induced by the scalar product.
We have
\[
\qform{\varphi^\vee(\alpha),x}=2(\alpha,x)=(\alpha,\alpha)\cdot\alpha^\vee(x)=d_{\alpha^\vee}\cdot\alpha^\vee(x),
\]
hence $\varphi^\vee(\alpha)=d_{\alpha^\vee}\cdot\alpha^\vee$.  And similarly for $\varphi$.  For uniqueness of $\varphi$ and $\varphi^\vee$, it suffices to note that the short roots generate $V^*$, which is obvious because they generate a subspace that is invariant under the Weyl group.

Let $x\in \Lambda_w$. By definition,
\[
\Z\ni \alpha^\vee(x)=\frac{2(x,\alpha)}{(\alpha,\alpha)}
\]
for all $\alpha\in R$.
It follows that $\qform{\varphi^\vee(x), \alpha} =2(x,\alpha)\in\Z$ since $(\alpha,\alpha)\in\Z$.
Therefore, $\varphi(x)\in\Lambda_w^\vee$.

For each root $\beta \in R$, $\varphi^\vee \varphi(\beta^\vee) = d_{\beta} d_{\beta^\vee} \beta^\vee$ and similarly for $\varphi \varphi^\vee$.  As $R$ is irreducible, either all roots have the same length (in which case $d_\beta d_{\beta^\vee} = 1$) or there are two lengths and $\beta$ and $\beta^\vee$ have different lengths (in which case $d_\beta d_{\beta^\vee}$ is the square-length of a long root); in either case the product equals $d_\alpha$ as claimed.
\end{proof}

\begin{remark}
If the root system $R$ is simply laced, then $\varphi$ gives isomorphisms from $V^*$, $\Lambda_w^\vee$, and $\Lambda_r^\vee$ to $V$, $\Lambda_w$, and $\Lambda_r$ respectively that agree with the canonical bijection $R^\vee \to R$ defined by $\alpha^\vee \leftrightarrow \alpha$.
\end{remark}

\begin{example} \label{phi.f}
For $\alpha^\vee$ a simple coroot, we write $f^\vee_\alpha$ for the corresponding fundamental dominant weight of $R^\vee$.  Consider an element $x' = \sum x_\beta \beta^\vee$ where $\beta$ ranges over the simple roots.  As $(f_\alpha^\vee, \beta^\vee)^\vee = \frac12 \qform{f_\alpha^\vee, \beta}(\beta^\vee, \beta^\vee)^\vee$, we have $(f_\alpha^\vee, x')^\vee = x_\alpha(f_\alpha^\vee, \alpha^\vee)^\vee = \frac12 d_\alpha x_\alpha$.  That is, $\qform{\varphi(f_\alpha^\vee), x'} = d_\alpha x_\alpha = \qform{d_\alpha f_\alpha, x'}$ for all $x'$, and we conclude that $\varphi(f_\alpha^\vee) = d_\alpha f_\alpha$.
\end{example}

\begin{remark}
Let $q\in \cat{S}^2(\Lambda_w)^W$ be the only quadratic form on $\Lambda_r^\vee$ that is equal to $1$ on
every short coroot in every component of $R^\vee$. It is shown in \cite[Lemma 2.1]{Merkurjev13} that the polar form $p$ of $q$ in $\Lambda_w\tens \Lambda_w$
in fact belongs to $\Lambda_r\tens \Lambda_w$. Then the restriction of $\varphi$ on $\Lambda_w^\vee$ coincides with
the composition
\[
\Lambda_w^\vee \xrightarrow{\id \tens p}
\Lambda_w^\vee \tens (\Lambda_r\tens \Lambda_w)=
(\Lambda_w^\vee \tens \Lambda_r)\tens \Lambda_w \to \Lambda_w.
\]
\end{remark}

\subsection{The map $\rho$} \label{rho.def}
Write $\Delta:=\Lambda_w/\Lambda_r$ and $\Delta^\vee:=\Lambda_w^\vee/\Lambda_r^\vee$. Note that $\Delta$ and $\Delta^\vee$ are dual
to each other with respect to the pairing
\[
\Delta\tens \Delta^\vee\to \Q/\Z.
\]

The group $W$ acts trivially on $\Delta$ and $\Delta^\vee$, hence $\Delta$ and $\Delta^\vee$ are $\Aut(\Dyn(R))$-modules.
The homomorphism $\varphi$ yields an $\Aut(\Dyn(R))$-equivariant homomorphism
\[
\rho: \Delta^\vee \to \Delta.
\]

The map $\rho$ is an isomorphism if $R$ is simply laced (because $\varphi$ is an isomorphism) or if $\Lambda_w = \Lambda_r$.
Similarly, $\rho = 0$ iff $\varphi(\Lambda_w^\vee) \subseteq \Lambda_r$, iff $p \in \Lambda_r \tens \Lambda_r$.

\begin{example} \label{rho.C}
Suppose $R$ has type $\cat{C}_n$ for some $n \ge 3$.  Consulting the tables in \cite{Bou:g4}, $f_n^\vee$, the fundamental weight of $R^\vee$ dual to the unique long simple root $\alpha_n$, is the only fundamental weight of $R^\vee$ not in the root lattice.  As $\alpha_n$ is long, $d_{\alpha_n} = 1$, so $\varphi(f^\vee_n) = f_n$, which belongs to $\Lambda_r$ iff $n$ is even.  That is, $\rho = 0$ iff $n$ is even; for $n$ odd, $\rho$ is an isomorphism.
\end{example}

\begin{example} \label{rho.B}
Suppose $R$ has type $\cat{B}_n$ for some $n \ge 2$.  For the unique short simple root $\alpha_n$, $d_{\alpha_n} = 2$, and $\varphi(f_n^\vee) = 2f_n$ is in $\Lambda_r$.  For $1 \le i < n$, $\varphi(f_i^\vee) = f_i \in \Lambda_r$.  We find that $\rho = 0$ regardless of $n$.
\end{example}

Thus we have determined $\ker \rho$ for every irreducible root system.

\begin{example}
Suppose $R$ is irreducible and $\car k = d_\alpha$ for some short root $\alpha$.  Then for $G$, $G^\vee$ simple simply connected with root system $R$, $R^\vee$ respectively, there is a ``very special'' isogeny $\pi \!: G \to G^\vee$.  The restriction of $\pi$ to a maximal torus in $G$ induces a $\Z$-linear map  on the cocharacter lattices $\pi_* \!: \Lambda_r^\vee  \to \Lambda_r$, which, by \cite[Prop.~7.1.5]{CGP2} or \cite[10.1]{St:rep}, equals $\varphi$.

In case $R = \cat{B}_n$,  $\pi$ is the composition of the natural map $G = \Spin_{2n+1} \to \SO_{2n+1}$ with the natural (characteristic 2 only) map $\SO_{2n+1} \to \Sp_{2n}$.  As $\pi$ vanishes on the center of $G$, it follows that $\rho = 0$ as in Example \ref{rho.B}.  Similarly, in case $R = \cat{C}_n$, one can recover Example \ref{rho.C} by noting that the composition $\pi \!: G = \Sp_{2n} \to \Spin_{2n+1}$ with the spin representation $\Spin_{2n+1} \injects \GL_{2^n}$ is the irreducible representation of $G$ with highest weight $f_n$ by \cite[\S11]{St:rep}.
\end{example}

\begin{example} \label{rho.A}
For $R = \cat{A}_{n-1}$, define $\psi \!: \Delta \iso \Z/n\Z$ via $\psi(f_1) = 1/n \in \Q/\Z$.  As $\qform{f_1, f_1^\vee} = (n-1)/n \in \Q$, defining $\psi^\vee \!: \Delta^\vee \iso \Z/n\Z$ via $\psi^\vee(f_1) = -1/n \in \Q/\Z$ gives a commutative diagram
\[
\xymatrix{\Delta \ot \Delta^\vee \ar[r]^{\qform{\ , \ }} \ar[d]_{\psi \ot \psi^\vee}^{\wr}& \Q/\Z \\
\Z/n\Z \ot \Z/n\Z \ar[ur]_{\mathrm{natural}}}
\]
i.e., $\psi^\vee$ is the isomorphism induced by $\psi$ and the natural pairings.  Furthermore, although $\rho$ is induced by the canonical isomorphism $R^\vee \simeq R$, the previous discussion shows that the diagram
\begin{equation} \label{rho.minus}
\xymatrix{\Delta^\vee \ar[r]^\rho \ar[d]_{\psi^\vee}^{\wr} & \Delta \ar[d]_{\psi}^{\wr} \\
\Z/n\Z \ar[r]^{-1} & \Z/n\Z}
\end{equation}
commutes, where the bottom arrow is multiplication by $-1$.

(The action of $\Aut(R)$ on $\Delta$ interchanges $f_1$ and $f_{n-1}$.  Defining instead $\psi(-f_1) = \psi(f_{n-1}) = 1/n$, also gives the same commutative diagram \eqref{rho.minus}.  That is, the commutativity of \eqref{rho.minus} is invariant under $\Aut(R)$.)
\end{example}

%%%%%%%%%%%%%%%%%%%%%%%%%%%%%%%%%%%%%%%%%%%%%%%%%%%%%%%%%%%%
\section{Statement of the main result}

\subsection{The map $\rho$}
Let $G$ be a simply connected semisimple group with root system $R$.
Let $C$ be the center of $G$. Then $C^*=\Lambda_w/\Lambda_r=\Delta$ and $C_*=\Lambda_w^\vee/\Lambda_r^\vee=\Delta^\vee$,
and we get from \S\ref{rho.def} a
homomorphism
\[
\rho=\rho_G: C_*\to C^*
\]
of Galois modules.
Therefore, we have a group homomorphism
\[
\hat\rho=\hat\rho_G: C \to C^\circ.
\]
Note that $\hat\rho$ is an isomorphism if $R$ is simply laced.

\subsection{The Tits class}
Let $G$ be a simply connected group over $F$ with center $C$.
Write $t_{G}$ for the \emph{Tits class} $t_{G}\in H^2(F,C)$. By definition, $t_G=-\partial(\xi_G)$, where
\[
\partial:H^1(F,G/C) \to H^2(F,C)
\]
is the connecting map for the exact sequence $1\to C \to G \to G/C \to 1$ and $\xi_G \in H^1(F,G/C)$ is the unique
 class such that the twisted group ${}^\xi G$ is quasi-split.

\subsection{Rost invariant for an absolutely simple group} \label{charp.sign}
Let $G$ be a simply connected group over $F$.
Recall (see \cite{GMS}) that, for $G$ absolutely simple, Rost defined an invariant $r_G \in \Inv^3(H^1(G))$ called the \emph{Rost invariant}, i.e., a map
\[
r_G: H^1(F,G)\to H^3(F,\Q/\Z(2))
\]
that is functorial in $F$.

\begin{lemma} \label{org}
If $G$ is an absolutely simple and simply connected algebraic group, then $o(r_G) \cdot t_G = 0$.
\end{lemma}

\begin{proof}
The order $o(r_G)$ of $r_G$ is calculated in \cite{GMS}, and in each case it is a multiple of the order of $t_G$.
\end{proof}

As mentioned in \cite[\S 2.3]{Gille00}, there are several definitions of the Rost invariant that may differ
by a sign. In \cite{GQ11}, Gille and Qu\'eguiner proved that for the definition of the Rost invariant $r_G$  they choose,
in the case $G=\gSL_1(A)$ for a central simple algebra $A$ of degree $n$ over $F$, the value of $r_G$ on the image of the class
$a F^{\times n}\in F^\times/F^{\times n}=H^1(F,\mu_n)$ in $H^1(F,G)$ is equal to $(a)\cup[A]$ if $n$ is not divisible
by $\ch(F)$ and to $-(a)\cup[A]$ if $n$ is a power of $p=\ch(F)>0$. We propose to normalize the Rost invariant
by multiplying the $p$-primary component of the Rost invariant (of all groups) by $-1$ in the case $p=\ch(F)>0$.
Thus the formula $r_G(a F^{\times n})=(a)\cup[A]$ holds now in all cases.

\subsection{The main theorem}
For $G$ semisimple and simply connected over $F$, there is an isomorphism 
\begin{equation} \label{G.sc}
\psi\!: G \iso \prod_{i=1}^n R_{F_i/F}(G_i), 
\end{equation} 
where the $F_i$ are finite separable extensions of $F$ and $G_i$ is an absolutely simple and simply connected $F_i$-group.  The product of the corestrictions of the $r_{G_i}$ (in the sense of \cite[p.~34]{GMS}) is then an invariant of $H^1(G)$, which we also denote by $r_G$ and call the Rost invariant of $G$.  The map $\psi$ identifies the center $C$ of $G$ with $\prod_i R_{F_i/F}(C_i)$ for $C_i$ the center of $G_i$, and the Tits class $t_G \in H^2(F, C)$ with $\sum t_{G_i} \in \sum H^2(F_i, C_i)$.

The composition $r_G\circ i^*$ is a group  homomorphism by \cite[Corollary 1.8]{MPT02} or \cite[Lemma 7.1]{G:rinv}.
That is, the composition $r_G\circ i^*$ in Theorem \ref{main} taken over all field extensions of $F$
can be viewed not only as an invariant of $H^1(C)$, but as an element of $\Inv^3_h(H^1(C)$ as in Definition \ref{inv.def}. Over a separable closure of $F$, the inclusion of $C$ into $G$ factors through
a maximal split torus and hence this invariant is trivial by Hilbert Theorem 90.
By Proposition \ref{invh1} the composition is given by the cup product with a unique element in $H^2(F,C^\circ)$. We will prove Theorem \ref{main}, which says that this element is equal to $-t_G^\circ$.

\subsection{Alternative formulation}
Alternatively, we could formulate the main theorem as follows.  The group of invariants $\Inv^3(H^1(G))$ is a sum of $n$ cyclic groups with generators (the corestrictions of) the $r_{G_i}$'s, and in view of Lemma \ref{org} we may define a homomorphism
\begin{equation} \label{inv.h2}
\Inv^3(H^1(G)) \to H^2(F, C) \quad \text{via $\sum n_i r_{G_i} \mapsto \sum -n_i t_{G_i}$.}
\end{equation}

\begin{theorem} \label{main2}
For every invariant $s \!: H^1(*, G) \to H^3(*, \Q/\Z(2))$, the composition
\[
H^1(*, C) \to H^1(*, G) \to H^3(*, \Q/\Z(2))
\]
equals the cup product with the image of $s$ under the composition  
\[
\Inv^3(H^1(G)) \to H^2(F, C) \to H^2(F, C^\circ).
\]
\end{theorem}

This will follow immediately from the main theorem, which we will prove over the course of the next few sections.

%%%%%%%%%%%%%%%%%%%%%%%%%%%%%%%%%%%%%%%%%%%%%%%%%%%%%%%%%%%%%%%%%%%%%
\section{Rost invariant of transfers} 

%The \'etale cohomology groups used were defined in \S\ref{cohom}.
The following statement is straightforward.

\begin{lemma}\label{corandcirc}
Let $A$ be an \'etale $F$-algebra and $G$ a simply connected semisimple group scheme over $A$, $C$ the center of $G$. Then $C':=R_{A/F}(C)$
is the center of $G':=R_{A/F}(G)$ and ${C'}^\circ\simeq R_{A/F}(C^\circ)$, and the diagram
\[
\xymatrix{
H^i(A, C)\ar[d]^\wr\ar[r]^{\hat\rho_G^*} & H^i(A, C^\circ)\ar[d]^\wr  \\
H^i(F, C') \ar[r]^{\hat\rho_{G'}^*} & H^i(F, {C'}^\circ)
 }
\]
commutes.
\end{lemma}
\begin{lemma}\label{cortits}
Let $G$ be a simply connected semisimple group scheme over an \'etale $F$-algebra $A$.
Set $C':=R_{A/F}(C)$
and $G':=R_{A/F}(G)$. Then
the image of $t_G$ under the isomorphism $H^2(A, C)\iso H^2(F, C')$ is equal to $t_{G'}$.
\end{lemma}
\begin{proof}
The corestriction of a quasi-split group is quasi-split.
\end{proof}

\begin{lemma}\label{reduction}
Let $G$ be a simply connected semisimple algebraic group scheme over an \'etale $F$-algebra $A$.  If Theorem \ref{main}
holds for $G$, then it also holds for $R_{A/F}(G)$.
\end{lemma}
\begin{proof}
Let $C$ be the center of $G$ and $G':=R_{A/F}(G)$. The group $C':=R_{A/F}(C)$ is the center of $G'$. Let $x\in H^1(A, C)$ and let $x'\in H^1(F, C')$
be the image of $x$ under the isomorphism $\nu:H^1(A, C)\iso H^1(F, C')$.
We have
\begin{align*}
r_{G'}({i'}^*(x')) & = r_{G'}(\nu({i}^*(x)))  \\
& = N_{A/F}(r_{G}({i}^*(x)) \quad\text{by \cite[Proposition 9.8]{GMS}} \\
& = N_{A/F}(-t_G^\circ\cup x) \quad\text{by Theorem \ref{main} for $x$}  \\
& = -t_{G'}^\circ\cup x'  \quad\text{by Lemmas \ref{compac}, \ref{corandcirc}  and \ref{cortits}.}\qedhere
\end{align*}
\end{proof}

If Theorem \ref{main} holds for semisimple groups $G_1$ and $G_2$, then it also holds for the group $G_1 \times G_2$. Combining this with Lemma \ref{reduction}, reduces the proof of Theorem \ref{main} to the case where $G$ is absolutely almost simple.

%%%%%%%%%%%%%%%%%%%%%%%%%%%%%%%%%%%%%%%%%%%%%%%%%%%%%%%%%%%%%%%

\section{Rost invariant for groups of type $\cat{A}$}\label{typa}

In this section, we will prove Theorem \ref{main} for $G$ absolutely simple of type $\cat{A}_{n-1}$ for each $n \ge 2$.

\subsection{Inner type} Suppose $G$ has inner type.  Then there is an isomorphism $\psi \!: G \iso \gSL_1(A)$, where $A$ is a central simple algebra of degree $n$ over $F$.  The map $\psi$ restricts to an isomorphism $C \iso\mu_n$, identifying $C^*$ with $\Z/n\Z$, and induces $\psi^\circ \!: C^\circ \iso \mu_n$.

The connecting homomorphism arising from the Kummer sequence $1 \to \mu_n \to \Gm \to \Gm \to 1$ gives an isomorphism $H^1(K, \mu_n) \simeq K^\times / K^{\times n}$ for every extension $K/F$.  

For each field extension $K/F$, the isomorphism $\psi$ identifies the map $H^1(K, C) \to H^1(K, G)$ with the obvious map
$K^\times/K^{\times n} = H^1(K, \mu_n) \to H^1(K, \gSL_1(A)) = K^\times/\Nrd(A_K^\times)$.  Further, $\psi(t_G) \in H^2(K, \mu_n)$ is the Brauer class $[A]$ of $A$ as in \cite[pp.~378, 426]{Book}.  The commutative diagrams in Example \ref{rho.A} show that, for $c \in H^1(K, C)$, the cup product $-t_G^\circ \cup c$ equals the cup product $-\psi^\vee(t_G^\circ) \cup \psi(c)$, where $-\psi^\vee(t_G^\circ) = \psi(t_G) = [A]$.  That is, 
Theorem \ref{main} claims that the composition
\begin{equation} \label{rost.A}
H^1(K, \mu_n) \xrightarrow{\psi^{-1}} H^1(K, C) \to H^1(K, G) \to H^3(K, \Q/\Z(2))
\end{equation}
equals the cup product with $[A]$.  

Let $p$ be a prime integer and write $m$ for the largest power of $p$ dividing $n$.  As the composition \eqref{rost.A} is a group homomorphism, it suffices to verify Theorem \ref{main} on each $p$-primary component $r_G(x)_p$ of the Rost invariant with values in $\Q_p/\Z_p(2)$.

Suppose first that $p$ does not divide $\ch(F)$. In this case, 
it was determined in \cite[Theorem 1.1]{GQ11} that the $p$-component  of the Rost invariant of $G$
is given by the formula
\[
r_G(x)_p=[A]_p\cup (x) \in H^3(K, \Q_p/\Z_p(2))
\]
for every $x\in K^\times$, where the cup product is taken with respect to the natural pairing
\[
H^2(K, \mu_m)\tens H^1(K, \mu_m) \to H^3(K, \Q_p/\Z_p(2))
\]
induced by $\mu_n\tens\mu_n\to\Q/\Z(2)$
and $[A]_p$ is the $p$-primary component of the class of $A$ in the Brauer group $\Br(K)=H^2(K, \Q/\Z(1))$. %Recall that the Tits class $t_G$ coincides with $[A]$.
Thus, Theorem \ref{main} holds in this case.

Now let $p=\ch(F)>0$. Consider the sheaf $\nu_m(j)$ in the \'etale topology over $F$ defined by
$\nu_m(j)(L)=K_j(L)$. The natural morphisms $\Z(j)\to\nu_m(j)[-j]$ for $j\leq 2$ are consistent with the products,
hence we have a commutative diagram
\[
\xymatrix{
(\Z/m\Z)(1)\tiL (\Z/m\Z)(1) \ar[d]^\wr\ar[r] & (\Z/m\Z)(2) \ar[d]^\wr \\
\nu_m(1)[-1]\tens \nu_m(1)[-1] \ar[r] & \nu_m(2)[-2].
 }
\]
Therefore, we have a commutative diagram
\[
\xymatrix{
H^2(F,\mu_m) \tens H^1(F,\mu_m) \ar[d]^\wr\ar[r] & H^3(F,\Q/\Z(2)) \ar[d]^\wr \\
H^1(F,\nu_m(1))\tens H^0(F,\nu_m(1)) \ar[r] & H^1(F,\nu_m(2))
 }
\]
(see Remark \ref{signs} after Proposition \ref{commute}). The bottom arrow is given by the cup product map
\[
\nBr{m}(F)\tens (F^\times/F^{\times m}) \to H^3(F,\Q/\Z(2))
\]
(see \cite[4D]{GQ11}). It is shown in \cite[Theorem 1.1]{GQ11} that the $p$-component  of the Rost invariant of $G$
is given by the formula
\[
r_G(x)_p= [A]_p\cup (x)=\theta^*([A]_p\cup (x)) \in H^3(K, \Q_p/\Z_p(2))
\]
for every $x\in K^\times$.  (The formula in \cite{GQ11} contains an additional minus sign, but it does not appear here due to the adjustment in the definition of $r_G$ in \S\ref{charp.sign}.)

%%%%%%%
\subsection{Outer type}  Now suppose that $G$ has outer type $\cat{A}_{n-1}$.  There is an isomorphism $\psi \!: G \iso \gSU(B,\tau)$,
where $B$ is a central simple algebra of degree $n$ over a separable quadratic field extension $K/F$ with
an involution $\tau$ of the second kind ($\tau$ restricts to a nontrivial automorphism of $K/F$).  The map $\psi$ identifies $C$ with $\mu_{n,[K]}$, and $C = C^\circ$.

Suppose first that $n$ is odd.
Since $G_K \simeq \gSL_1(B)$, the theorem holds over $K$.  As $K$ has degree 2 over $F$ and $C$ has odd exponent, the natural map $H^1(F, C) \to H^1(K, C)$ is injective, hence the theorem holds over $F$ by the following general lemma.

\begin{lemma}\label{fe}
Let $L_1,L_2,\dots, L_s$ be field extensions of $F$ such that the natural homomorphism $H^2(F,C)\to\prod_i H^2(L_i,C)$ is injective. If Theorem \ref{main}
holds for $G$ over all fields $L_i$, then it also holds over $F$.
\end{lemma}
\begin{proof}
The left vertical map in Theorem \ref{main} is multiplication by some element $h\in H^2(F,C^\circ)$.
We need to show that $h=-t_G^\circ$. This equality holds over all fields $L_i$, hence it holds over $F$ by the injectivity assumption.
\end{proof}

So we may assume that $n$ is even.
Then $H^1(F,C)$ is isomorphic to a factor group of the group of pairs $(a,z)\in F^\times \times K^\times$ such that $N_{K/F}(z)=a^n$
and $H^2(F,C)$ is isomorphic to a subgroup of $\Br(F)\oplus\Br(K)$
of all pairs $(v,u)$ such that $v_K=mu$ and $N_{K/F}(u)=0$, see \cite[pp.~795, 796]{MPT02}.

Suppose that $B$ is split; we follow the argument in \cite[31.44]{Book}.  Then $\gSU(B,\tau)=\gSU(h)$, where $h$ is a hermitian form of trivial discriminant
on a vector $K$-space of dimension $n$ for the quadratic extension $K/F$. Let
$q(v):=h(v,v)$ be the associated quadratic form on $V$ viewed as a $2n$-dimensional $F$-space. The quadratic form $q$
is nondegenerate, and we can view $\gSU(h)$ as a subgroup of $H:=\gSpin(V,q)$. The Dynkin index of $G$ in $H$ is $1$, hence the
the composition $H^1(K,G)\to H^1(K,H)\xra{r_H} H^3(K)$ equals the Rost invariant of $G$. Then $r_H$ is given by the Arason invariant.
A computation shows that the image of the pair $(a,z)$ representing an element $x\in H^1(F,C)$ under the composition
\[
H^1(F,C)\to H^1(F,G)\xra{r_G} H^3(F)
\]
coincides with $[D]\cup x$, where $D$ is the class of the discriminant algebra of $h$. On the other hand, $[D]\cup x$
coincides with the cup product of $x$ with the Tits class $t_G=-t_G^\circ$ represented by the pair $([D],0)$
 in $H^2(F,C^\circ)$, proving Theorem \ref{main} in this case.

Now drop the assumption that $B$ is split.  As for the $n$ odd case, the theorem holds over $K$.
Note that there is an injective map $H^2(F,C)\to  \Br(F)\oplus\Br(K)$. Let $X=R_{K/F}(SB(B))$. By \cite[2.4.6]{MT95},
the map $\Br(F)\to \Br(F(X))$ is injective,
hence the natural homomorphism
\[
H^2(F,C) \to H^2(F,C_{F(X)}) \oplus H^2(F,C_K)
\]
is injective. The theorem holds over $K$ and by the preceding paragraph the theorem holds over $F(X)$.  Therefore, by Lemma \ref{fe}, the theorem holds over $F$.

%%%%%%%%%%%%%%%%%%%%%%%%%%%%%%%%%%%%%%%%%%%%%%%%%%%%%%%%%%%%%%%%%%%%%
\section{Conclusion of the proof of Theorem \ref{main}}

Choose a system of simple roots $\Pi$ of $G$. Write $\Pi_r$ for the subset of $\Pi$ consisting of all simple roots
whose fundamental weight belongs to $\Lambda_r$ and let $\Pi':=\Pi\setminus \Pi_r$.
Inspection of the Dynkin diagram shows that all connected component of $\Pi'$ have type $\cat{A}$.

Every element of $\Pi_r$ is fixed by every
automorphism of the Dynkin diagram, hence is fixed by the $*$-action of the absolute Galois group of $F$ on $\Pi$. It follows that the
variety $X$ of parabolic subgroups of $G_\sep$ of type $\Pi'$ is defined over $F$.
By \cite{MT95}, the kernel of the restriction map $\Br(K)\to \Br(K(X))$ for every field extension $K/F$ is generated by the Tits algebras associated
with the classes in $C^*$ of the fundamental weights $f_\alpha$ corresponding to the simple roots $\alpha \in \Pi_r$. But $f_\alpha \in \Lambda_r$, so these Tits algebras are split \cite{Tits71}, hence
the restriction map $\Br(K)\to \Br(K(X))$ is injective and, by Lemma \ref{inj}, the natural homomorphism $H^2(F,C)\to H^2(F(X),C)$ is injective.
In view of Lemma \ref{fe}, it suffices to prove Theorem \ref{main} over the field $F(X)$, i.e., we
may assume that $G$ has a parabolic subgroup of type $\Pi'$. As shown in \cite{GQ07},
there is a simply connected subgroup $G'\subset G$ with Dynkin diagram $\Pi'$.
We write $R'\subset R$ for the root system of $G'$.

The center $C'$ of $G'$ contains $C$. Write $j$ for the embedding homomorphism $C\to C'$ and $j^\circ$ for the dual $C'^{\circ}\to C^\circ$.

Let $G'=\prod_i G'_i$ with $G_i$ simply connected simple groups, $C=\prod C_i$, where $C_i$ is the center of $G_i$, $\Pi'_i\subset \Pi$
the system of simple roots of $G_i$. Write $j_i^\circ$ for the composition $C_i'^{\circ}\to C'^{\circ}\to C^\circ$.

\begin{lemma}\label{11}
The map $j_i^*:H^2(F, C) \to H^2(F, C'_i)$ takes the Tits class $t_G$ to $t_{G'_i}$.
\end{lemma}
\begin{proof}[Proof \#1]
It suffices to check that $j^*(t_G) = t_{G'}$, for the projection $H^2(F, C') \to H^2(F, C'_i)$ sends $t_{G'} \mapsto t_{G'_i}$.

There is a rank $|\Pi_r|$ split torus $S$ in $G$ whose centralizer is $S \cdot G'$.
Arguing as in Tits's Witt-type Theorem \cite[2.7.1, 2.7.2(d)]{Tits66}, one sees that the quasi-split inner form of $G$ is obtained by twisting $G$ by a 1-cocycle $\gamma$ with values in $C_G(S)/C$, equivalently, in $G'/C$.  Clearly, twisting $G'$ by $\gamma$ gives the quasi-split inner form of $G'$.  The Tits class $t_G$ is defined to be $-\partial_G(\gamma)$ where $\partial_G$ is the connecting homomorphism $H^1(F, G/C) \to H^2(F, C)$ induced by the exact sequence $1 \to C \to G \to G/C \to 1$ and similarly for $G'$ and $C'$.  The diagram
\[
\xymatrix{
H^1(F, G'/C) \ar[r] \ar[d] & H^1(F, G/C) \ar[r]^{\partial_G} & H^2(F, C) \ar[d]^{j^*} \\
H^1(F, G'/C') \ar[rr]^{\partial_{G'}} && H^2(F, C')
}
\]
commutes trivially, so $j^*(t_G) = j^*(-\partial_G(\gamma)) = -\partial_{G'}(\gamma) = t_{G'}$ as claimed.
\end{proof}

\begin{proof}[Proof \#2]
For each $\chi \in T^*$, define $F(\chi)$ to be the subfield of $F_\sep$ of elements fixed by the stabilizer of $\chi$ under the Galois action.  Note that because $G$ is absolutely almost simple, the $*$-action fixes $\Pi_r$ elementwise, $F(\chi)$ equals the field extension $F(\chi\vert_{T'})$ defined analogously for $\chi \in (T')^*$.  The diagram
\[
\xymatrix{
H^2(F, C) \ar[r]^{j^*} \ar[rd]^{\chi\vert_C} & H^2(F, C') \ar[d]^{\chi\vert_{C'}} \\
&H^2(F(\chi), \gm)
}
\]
commutes.  Now $\chi\vert_{C'}(t_{G'} - j^*(t_G)) = \chi\vert_{C'}(t_{G'}) - \chi\vert_C(t_G)$, which is zero for all $\chi \in T^*$ by \cite[\S5.5]{Tits71}.  As $\prod_{\chi \in (T')^*} \chi\vert_{C'}$ is injective by \cite[Prop.~7]{G:outer}, $t_{G'} = j^*(t_G)$ as claimed.
\end{proof}

The diagram $\Pi'_i$ is simply laced. Write $d_i$ for the square-length of $\alpha^\vee \in R^\vee$ for $\alpha\in\Pi'_i$.

\begin{lemma}\label{2}
The homomorphism $\hat\rho_G:C\to C^\circ$ coincides with the composition
\[
C \xra{j} C' \xra{\hat\rho_{G'}} C'^{\circ} = \prod_i C_i'^{\circ} \xra{\prod_i ({{j}_i}^\circ)^{d_i}} C^\circ,
\]
where $j_i$ is the composition $C\to C'\to C'_i$.
\end{lemma}

\begin{proof}
For every simple root $\alpha\in\Pi$ write $f_\alpha$ for the corresponding fundamental weight.
Write $\Lambda'_r$ and $\Lambda'_w$ for the root and weight lattices respectively of the root system $R'$. Let
\[
\Phi:=\{f_\alpha,\alpha\in\Pi_r\}.
\]
Then $\Phi$ is a $\Z$-basis for the kernel  of the natural surjection $\Lambda_w\to \Lambda'_w$.
If $\alpha\in\Pi'$, we write $\alpha'$ for the image of $\alpha$ and $f'_\alpha$ for the image of $f_\alpha$
under this surjection. All $\alpha'$ (respectively, $f'_\alpha$) form the system of simple roots
(respectively, fundamental weights) of $R'$.
If $\alpha\in\Pi'$, the image ${\alpha'}^\vee$ of ${\alpha}^\vee$ under the inclusion ${\Lambda'_r}^\vee\hookrightarrow {\Lambda_r}^\vee$
is a simple coroot of $R'$.

If $V$ is the real vector space of $R$, then $R'\subset V':=V/\spp(\Phi)$ and ${R'}^\vee\subset {V'}^*\subset V^*$.
Let $x\in\Lambda_w^\vee$, i.e., $\langle x,\alpha\rangle\in\Z$ for all $\alpha\in\Pi$. Since $\Phi\subset\Lambda_r$, we have
$a_\alpha:=\langle x,f_\alpha\rangle\in\Z$ for all $\alpha\in \Pi_r$. Then the linear form $x':=x-\sum_{\alpha\in\Pi_r}a_\alpha\alpha^\vee$
vanishes on the subspace of $V$ spanned by $\Phi$, hence $x'\in {\Lambda'_w}^\vee$. We then have a well defined homomorphism
\begin{equation}\label{form}
s: \Lambda_w^\vee \to {\Lambda'_w}^\vee, \quad x\mapsto x'.
\end{equation}
If $\alpha\in\Pi'$, then $\langle x',\alpha\rangle=\langle x',\alpha'\rangle$. It follows that if
$x'=\sum_{\alpha\in\Pi}b_\alpha f_\alpha^\vee$ in $\Lambda_w^\vee$ for $b_\alpha=\langle x',\alpha\rangle\in\Z$, then
$x'=\sum_{\alpha\in\Pi'}b_\alpha {f'_\alpha}^\vee$ in ${\Lambda'_w}^\vee$.

Since $\Phi\subset\Lambda_r$, we have a surjective homomorphism
\[
{C'}^*=\Lambda'_w/\Lambda'_r=\Lambda_w/\spp(\Phi,\Pi')\to\Lambda_w/\Lambda_r=C^*
\]
dual to the inclusion of $C$ into $C'$. The dual homomorphism
\[
C_*=\Lambda_w^\vee/\Lambda_r^\vee \to {\Lambda'_w}^\vee/{\Lambda'_r}^\vee=C'_*
\]
is induced by $s$.

Consider the diagram
\[
\xymatrix{
 \Lambda_w^\vee\ar[d]_{s}\ar[r]^{\varphi} & \Lambda_w  \\
 {\Lambda'_w}^{\!\!\vee} \ar[r]^{\varphi'} & \Lambda'_w \ar[u]_{t},
 }
\]
where the map $t$ is defined by $t(f'_\alpha)=d_\alpha f_\alpha$ for all $\alpha\in\Pi'$ and the maps $\varphi$
and $\varphi'$ are defined in Proposition \ref{phi}.

It suffices to prove that $\Im(t\circ\varphi'\circ s - \varphi)\subset \Lambda_r$.

Consider the other diagram
\[
\xymatrix{
 \Lambda_w^\vee\ar[r]^{\rho} & \Lambda_w  \\
 {\Lambda'_w}^{\!\!\vee} \ar[u]^{t^\vee}\ar[r]^{\rho'} & \Lambda'_w \ar[u]_{t},
 }
\]
where $t^\vee({f'_\alpha}^\vee)=f_\alpha^\vee$ for all $\alpha\in\Pi'$. This diagram is commutative. Indeed, 
\[
(\rho\circ t^\vee)({f'_\alpha}^\vee)=\rho(f_\alpha^\vee)=d_\alpha f_\alpha=t(f'_\alpha)=(t\circ \rho')({f'_\alpha}^\vee),
\]
where the second equality is by Example \ref{phi.f}.
(Recall that the root system $R'$ is simply laced, hence $\rho'({f'_\alpha}^\vee)=f'_\alpha$.)

We claim that
\[
(t^\vee\circ s)(x)-x\in \spp(\Phi^\vee)+\Lambda_r^\vee
\]
for every $x\in \Lambda_w^\vee$, where $\Phi^\vee:=\{f_\alpha^\vee,\alpha\in\Pi_r\}$. Indeed, in the notation of \eqref{form} we have
\begin{align*}
(t^\vee\circ s)(x)-x&= t^\vee(x')-x= t^\vee(x')-x'-\sum_{\alpha\in\Pi_r}a_\alpha\alpha^\vee \\
&=t^\vee\(\sum_{\alpha\in\Pi'}b_\alpha {f'_\alpha}^\vee\)-\sum_{\alpha\in\Pi}b_\alpha f_\alpha^\vee
-\sum_{\alpha\in\Pi_r}a_\alpha\alpha^\vee\\
&=-\sum_{\alpha\in\Pi_r}b_\alpha f_\alpha^\vee-\sum_{\alpha\in\Pi_r}a_\alpha\alpha^\vee\in \spp(\Phi^\vee)+\Lambda_r^\vee.
\end{align*}
It follows from the claim that
\[
(t\circ \rho'\circ s)(x)-\rho(x)=(\rho\circ t^\vee \circ s)(x)-\rho(x)=\rho\((t^\vee\circ s)(x)-x\)\in\rho\(\spp(\Phi^\vee)+\Lambda_r^\vee\).
\]
As $\rho(f_\alpha^\vee)=d_\alpha f_\alpha\in\Lambda_r$ for all $f_\alpha\in\Phi$, this is contained in $\Lambda_r$, proving the claim.
\end{proof}

Lemmas \ref{11} and \ref{2} yield:
\begin{corollary}\label{comp}
$t_G^\circ=\sum_i d_i\cdot j_i^{\circ *}(t_{G_i'}^\circ)$.
\end{corollary}

\begin{lemma}\label{3}
The diagram
\[
\xymatrix{
 H^1(F, G')\ar[d]\ar@{=}[r] & \prod_i H^1(F, G_i') \ar[d]^{\sum d_i\cdot r_{G'_i}}\\
 H^1(F, G) \ar[r]^-{r_{G}} & H^{3}(F,\Q/\Z(2))
 }
\]
commutes.
\end{lemma}

\begin{proof}
The composition $H^1(F, G'_i)\to  H^1(F, G) \xra{r_G} H^{3}(F,\Q/\Z(2))$  coincides with the $k^{th}$ multiple of the Rost invariant
$r_{G'_i}$, where $k$ is the order of the cokernel of the map $Q(G)\to Q(G'_i)$ of infinite cyclic groups generated by
positive definite quadratic forms $q_G$ and $q_{G'_i}$ on the lattices of coroots normalized so that the forms have value $1$ on the short coroots (see \cite{GMS}).
Recall that all coroots of $G'_i$ have the same length, hence $q_{G'_i}$ has value $1$ on all coroots of $G'_i$. Therefore, $k$ coincides with
$d_i$, the square-length of all coroots of $G'_i$ viewed as coroots of $G$.
\end{proof}

Write each $G'_i = R_{L_i/F}(H_i)$ for $L_i$ a separable field
 extension of $F$ and $H_i$ a simply connected absolutely simple algebraic group of type $\cat{A}$ over $L_i$.
Theorem \ref{main} is proved for such groups in \S\ref{typa}.
By Lemma \ref{reduction}, Theorem \ref{main} holds for the group $G_i'$ and hence for $G'$.

Let $x\in H^1(F, C)$ and let $y \in H^1(F, G)$, $\prod x'_i \in H^1(F, C') = \prod H^1(F, C'_i)$ and $\prod y'_i \in \prod H^1(F, G'_i)$ denote its images under the natural maps.
We find
\begin{align*}
r_G(y) & = \sum_i d_i\cdot r_{G'_i}(y_i) \quad\text{by Lemma \ref{3}}  \\
& = \sum_i d_i\cdot (-t^\circ_{G'_i}\cup x'_i)  \quad\text{by the main theorem for all $G'_i$} \\
& = \sum_i d_i\cdot j_i^{\circ*}(-t_{G'_i}^\circ)\cup x  \quad\text{by Lemma \ref{1}}  \\
& = -t_{G}^\circ\cup x  \quad\text{by Corollary \ref{comp}.}  
\end{align*}
This completes the proof of Theorem \ref{main}.

%%%%%%%%%%%%%%%%%%%%%%%%%%%%%%%%%%%%%%%%%%%%%%%%%%%%
\section{Concrete formulas}

In this section, we deduce the explicit formulas for the restriction of the Rost invariant to the center given in \cite{MPT02} and \cite{GQ07}.  Note that the explicit formulas given in those references required ad hoc definitions for the cup product, whereas here the formulas are deduced from Theorem \ref{main}.

\subsection{The pairing induced by $\rho$} \label{rho.subsec}
The map $\rho$ defines a bilinear pairing $\Delta^\vee \ot \Delta^\vee \to \Q/\Z$ via
\begin{equation} \label{rho.pair}
\Delta^\vee \ot \Delta^\vee \xrightarrow{\id \otimes \rho} \Delta^\vee \ot \Delta \to \Q/\Z.
\end{equation}
We now determine this pairing for each simple root system $R$.

For $R$ with different root lengths, $\rho$ is zero and hence \eqref{rho.pair} is zero unless $R = \cat{C}_n$ for odd $n \ge 3$.  In that case (and also for $R = \cat{E}_7$), $\Delta \simeq \Z/2 \simeq \Delta^\vee$ and $\rho$ is the unique isomorphism, so \eqref{rho.pair} amounts to the product map $x \ot y \mapsto xy$.
Therefore we may assume that $R$ has only one root length. 

If $\Delta^\vee$ is cyclic, we pick a fundamental dominant weight $f^\vee_i$ that generates $\Delta^\vee$ and the pairing \eqref{rho.pair} is determined by the image of $f_i^\vee \ot f_i^\vee$, which is the coefficient (in $\Q/\Z$) of the simple root $\alpha^\vee_i$ appearing in the expression for $f^\vee_i$ in terms of simple roots.  By hypothesis on $R$, the canonical isomorphism $R \simeq R^\vee$ is an isometry, and in particular this coefficient is the same as the coefficient for $\alpha_i$ in $f_i$, which can be looked up in \cite{Bou:g4}.

For $R = \cat{A}_n$, we have $\Delta^\vee \simeq \Z/(n+1)$ generated by $f^\vee_1$ and $f^\vee_1 \ot f^\vee_1 \mapsto n/(n+1)$, cf.~Example \ref{rho.minus}.

For $R = \cat{D}_n$ for odd $n > 4$, $\Delta^\vee \simeq \Z/4$ generated by $f^\vee_n$ and $f^\vee_n \ot f^\vee_n \mapsto n/4$.

For $R = \cat{E}_6$, we have $\Delta^\vee \simeq \Z/3$ generated by $f^\vee_1$ and $f^\vee_1 \ot f^\vee_1 \mapsto 1/3$.

For $R = \cat{D}_n$ for even $n \ge 4$, $\Delta^\vee$ is isomorphic to $\Z/2 \oplus \Z/2$ generated by $f^\vee_{n-1}$, $f^\vee_n$.  The tables show that $f^\vee_{n-1} \ot f^\vee_{n-1}$ and $f^\vee_n \ot f^\vee_n$ map to $n/4$ whereas $f^\vee_n \ot f^\vee_{n-1}$ and $f^\vee_{n-1} \ot f^\vee_n$ map to $(n-2)/4$.  That is, viewing \eqref{rho.pair} as a bilinear form on $\F_2 \oplus \F_2$, for $n \equiv 0 \bmod 4$ it is the wedge product (which is hyperbolic) and for $n \equiv 2 \bmod 4$ it is the unique (up to isomorphism) metabolic form that is not hyperbolic.

\subsection{The cup product on $C$} \label{C.cup}
Let $G$ be a simple simply connected algebraic group over $F$ with center $C$.
The pairing (\ref{rho.pair}) reads as follows:
\[
C_* \ot C_* \xrightarrow{\id \otimes \rho} C_* \ot C^* \to \Q/\Z.
\]
Twisting (tensoring with $\Z(1)\tiL \Z(1)$) we get a composition
\[
C_*(1)\stackrel{L}\tens C_*(1) \to C_*(1)\stackrel{L}\tens C^*(1) \to \Q/\Z(2),
\]
where the second map was already defined in (\ref{fp1}). Therefore, we have a pairing
\begin{equation}\label{newpair}
H^1(F,C)\tens H^2(F,C) \to H^1(F,C)\tens H^2(F,C^\circ) \to H^3(F,\Q/\Z(2)).
\end{equation}
This pairing takes $x\tens t_G$ to $x\cup t_G^\circ$ for every $x\in H^1(F,C)$. Thus, Theorem \ref{main} states that
the composition $r_G\circ i^*$ takes $x$ to $x\cup t_G$, where the cup-product is taken with respect to the pairing (\ref{newpair}) induced by (\ref{rho.pair}).  Combining this observation with the computation of \eqref{rho.pair} recovers the formulas given in \cite{MPT02} and \cite{GQ07}.

\def\cprime{$'$} \def\cprime{$'$}
\providecommand{\bysame}{\leavevmode\hbox to3em{\hrulefill}\thinspace}
\providecommand{\MR}{\relax\ifhmode\unskip\space\fi MR }
% \MRhref is called by the amsart/book/proc definition of \MR.
\providecommand{\MRhref}[2]{%
  \href{http://www.ams.org/mathscinet-getitem?mr=#1}{#2}
}
\providecommand{\href}[2]{#2}

\end{document}